\documentclass[final]{siamltex}
\usepackage{amssymb,epsfig,multirow}
\usepackage{graphicx}
\usepackage{float}
 \usepackage{color}

\newtheorem{lem}{Lemma}[section]

\newtheorem{thm}{Theorem}[section]
\newtheorem{cor}{Corollary}[section]

\newtheorem{exam}{Example}[section]

\newtheorem{rem}{Remark}[section]

\title{On the Ball-Constrained Weighted Maximin Dispersion Problem
\thanks{This research was supported by National Natural Science Foundation of China under grants
11471325 and 11571029, by Beijing Higher Education Young Elite Teacher Project 29201442 and by fundamental research funds for the Central Universities under
grant  YWF-15-SXXY-006.}}

\author{Shu Wang\footnotemark[2]\and Yong  Xia \footnotemark[2]}

\begin{document}
\maketitle
\renewcommand{\thefootnote}{\fnsymbol{footnote}}

\footnotetext[2]{State Key Laboratory of Software Development
              Environment, LMIB of the Ministry of Education, School of
Mathematics and System Sciences, Beihang University, Beijing,
100191, P. R. China ({\tt wangshu.0130@163.com; dearyxia@gmail.com})}

\begin{abstract}
The ball-constrained weighted maximin dispersion problem $(\rm P_{ball})$ is to find a point in an $n$-dimensional Euclidean ball such that the minimum of the weighted Euclidean distance from given $m$ points is maximized. We propose a new second-order cone programming relaxation for $(\rm P_{ball})$. Under the condition $m\le n$, $(\rm P_{ball})$ is polynomial-time solvable since the new relaxation is shown to be tight. In general, we prove that $({\rm P_{ball}})$ is NP-hard. Then, we propose a new  randomized approximation algorithm for solving $({\rm P_{ball}})$, which provides a new approximation bound of $\frac{1-O(\sqrt{\ln(m)/n})}{2}$.
\end{abstract}

\begin{keywords}
 maximin dispersion, convex relaxation, second-order cone programming, approximation algorithm
\end{keywords}

\begin{AMS}
90C22, 90C26, 90C59
\end{AMS}

\pagestyle{myheadings} \thispagestyle{plain} \markboth{S. Wang, Y.
Xia}{Ball-Constrained Maxmin Dispersion Problem}

\section{Introduction}
Consider the following weighted maximin dispersion problem
\begin{eqnarray*}
{\rm (P_{\chi})}~~\max_{x\in \Bbb \chi}\left\{ f(x):=\min_{i=1,\ldots,m} \omega_i\|x-x^i\|^2\right\},
\end{eqnarray*}
where $\chi=\{y\in \Bbb R^n\mid (y_1^2,\ldots,y_n^2,1)^T\in \mathcal{K}\}$, $\mathcal{K}$ is a convex cone, $x^1,\ldots,x^m\in \Bbb R^n$ are given $m$ points, $\omega_i>0$ for $i=1,\ldots,m$ and $\|x\|=\sqrt{x^Tx}$ is the Euclidean norm. Let $v{\rm (P_{\chi})}$ denote the optimal value of the problem  $({\rm P_{\chi}})$. Without loss of generality, we assume that $v{\rm (P_{\chi})}>0$, since $v{\rm (P_{\chi})}=0$ if and only if
$\chi\subseteq\{x^1,\ldots,x^m\}$. When all weights are equal, intuitively, the problem ${\rm (P_{\chi})}$ is to find the largest sphere centering in $\chi$ and enclosing none of the given $m$ points. The maximin dispersion problem has many applications in facility location, spatial management and pattern recognition, see \cite{DW,JMY,Sc,W} and references therein.

In a very recent paper \cite{HA13}, an approximation bound for ${\rm (P_{\chi})}$ is established based on the following  semidefinite programming relaxation (SDP$_{\chi}$) and second-order cone programming relaxation (SOCP$_{\chi}$), respectively:
\begin{eqnarray}
~~ \frac{1}{v({\rm SDP_{\chi}})}=&\min_{Z\in S^{n+1}}&~ Z_{n+1,n+1}\label{sdp}\\
&{\rm s.~t.}& \omega_i {\rm {Tr}}(A^iZ)\geq 1,~i=1,\ldots,m,\nonumber\\
&& (Z_{11},\ldots,Z_{nn},Z_{n+1,n+1})^T\in \mathcal{K},\nonumber\\
&& Z\succeq 0,\nonumber
\end{eqnarray}
\begin{eqnarray}
~~\frac{1}{v({\rm SOCP_{\chi}})}=&\min_{Z\in S^{n+1}}&~ Z_{n+1,n+1}\label{socp}\\
&{\rm s.~t.}& \omega_i {\rm {Tr}}(A^iZ)\geq 1,~i=1,\ldots,m,\label{socp:1}\\
&& (Z_{11},\ldots,Z_{nn},Z_{n+1,n+1})^T\in \mathcal{K},\label{socp:2}\\
&& \left\|\left(\begin{array}{c} 2Z_{j,n+1}\\ Z_{jj}-Z_{n+1,n+1}\end{array}\right)\right\|\leq Z_{jj}+Z_{n+1,n+1},j=1,\ldots,n,\label{socp:3}
\end{eqnarray}
where
$
A^i=\left(\begin{array}{cc}I & -x^i\\-(x^i)^T & \|x^i\|^2\end{array}\right)$ with $I$ being the identity matrix of order $n$,    ${\rm Tr}(AB^T)=\sum_{i=1}^n\sum_{j=1}^na_{ij}b_{ij}$ is
the inner product of two matrices $A$ and $B$, and $Z\succeq 0$ means that $Z$ is positive semidefinite.

Let $Z^*\in S^{n+1}$ be an optimal solution of the above SDP (or SOCP) relaxation,
 where $S^{n+1}$ denotes the set of $(n+1)\times (n+1)$ real symmetric matrices. Haines et al. \cite{HA13} find an approximate solution $\widetilde{x}\in\chi$ satisfying
 \begin{eqnarray*}
&& v({\rm SDP_{\chi}})\geq v({\rm P_{\chi}})\ge f(\widetilde{x})\geq \frac{1- \sqrt{2\ln(m/\rho)\gamma_1^*}}{2}\cdot v({\rm SDP_{\chi}}), \\
{\rm or}~&& v({\rm SOCP_{\chi}})\geq v({\rm P_{\chi}})\ge f(\widetilde{x})\geq \frac{1- \sqrt{2\ln(m/\rho)\gamma_1^*}}{2}\cdot v({\rm SOCP_{\chi}}),
 \end{eqnarray*}
 where   $0<\rho< 1$ and
  \begin{equation}
  \gamma_1^*=\frac{\max_{j=1,\ldots,n}Z^*_{jj}}{\sum_{j=1}^nZ^*_{jj}}.\label{gamma:0}
  \end{equation}
Such a factor as $\frac{1- \sqrt{2\ln(m/\rho)\gamma_1^*}}{2}$ is called an approximation bound in this paper. Notice that the approximation algorithm used in Haines et al. \cite{HA13} is a randomized one. We refer to \cite{HJLZ14,KN08} for more general randomized algorithms for polynomial optimization over binary variables or hyperspheres.

Two typical cases of $\chi$ are considered in \cite{HA13}. The first one is
$\chi=[-1,1]^n$, which corresponds to setting
\begin{equation}
\mathcal{K}=\{y\in \Bbb R^{n+1}\mid y_j\leq y_{n+1},~j=1,\ldots,n\}.\label{k:box}
\end{equation}
We denote this case of $({\rm P_{\chi}})$ by (${\rm P_{box}}$). Then, it has been shown in \cite{HA13} that (${\rm P_{box}}$) is NP-hard. Moreover, in this case, $\gamma_1^*$ defined in (\ref{gamma:0}) can be simplified to
$\gamma_1^*=\frac{1}{n}$,
which no longer depends on $Z^*$.
It immediately follows that (${\rm P_{box}}$) admits a $1/2$ asymptotic approximation bound as $\frac{n}{\ln (m)}$ goes to infinity.

The other case is
$\chi=\{x\mid \|x\|\leq1\}$, which corresponds to letting
\begin{equation}
\mathcal{K}=\{y\in \Bbb R^{n+1}\mid y_1+...+y_n\leq y_{n+1}\}.
\label{k:ball}
\end{equation}
Accordingly, this special case of $({\rm P_{\chi}})$ is denoted by (${\rm P_{ball}}$). However, the results similar to those for (${\rm P_{box}}$) 
remain unknown for (${\rm P_{ball}}$). Actually, they correspond to the three open questions (i),(ii),(vi) raised in the concluding section in \cite{HA13}. For convenience, we write in the following these questions relating to (${\rm P_{ball}}$):

(i) Can we extend the approximation bound $\frac{1-O(\sqrt{\ln(m)/n})}{2}$ to (${\rm P_{ball}}$)?

(ii) Is (${\rm P_{ball}}$) NP-hard?

(iii) When the SDP (or SOCP)
relaxation is tight?

This becomes the first motivation to give a comprehensive study on ${\rm (P_{ball})}$. On the other hand, ${\rm (P_{ball})}$ has its own applications.
Recently, Tenne et al. \cite{Te} propose  a model-classifier framework for solving optimization problems whose objective function values have to be evaluated by computationally expensive simulation. The framework employs a trust-region approach \cite{Conn}. It tries to make the interpolation model accurate enough in the trust region. Otherwise, a new point is selected from the trust region to add into the set of evaluated points for improving the interpolation model.
According to the theory on  the relation between the model accuracy and the distance among the training points \cite{Sc},
the new added point is remote from the existing evaluated points. The formulation of this subproblem is exactly the same as ${\rm (P_{ball})}$, see \cite{Te}.

In this paper, we focus on the study of ${\rm (P_{ball})}$ and answer the above three questions. Some of the new results can be extended to ${\rm (P_{box})}$. Firstly, we give a new convex relaxation for the general ${\rm (P_{\chi})}$, denoted by $({\rm CR_\chi})$. We show that $v({\rm CR_\chi})=v({\rm SDP_\chi})=v({\rm SOCP_\chi})$ when $\chi$ is a unit box or a unit ball. Besides, in the former case, $({\rm CR_\chi})$ reduces to a linear programming problem. In the latter case, $({\rm CR_\chi})$ is a new second-order cone programming problem, whose size is much smaller than that of $({\rm SOCP_\chi})$. Then, we derive a new sufficient condition under which our new convex relaxation for ${\rm (P_{ball})}$ is tight. Moreover, the new condition holds when $m\le n$. It strictly improves the existing sufficient condition $m\le n-1$ proposed in \cite{WX}.
Secondly, we prove that ${\rm (P_{ball})}$ remains NP-hard when $m\ge 2n$. Thirdly, we propose a new approximation algorithm for ${\rm (P_{ball})}$ based on uniform random sampling on the surface of the ball. Using a new tail estimation,
we show that the approximation bound of our new algorithm is $\frac{1-O(\sqrt{\ln(m)/n})}{2}$. It should be noticed that our approximation algorithm does not rely on any convex relaxation.

The remainder is organized as follows. In Section 2, we study a new convex relaxation for ${\rm (P_{\chi})}$. In Section 3, we show that ${\rm (P_{ball})}$ is NP-hard and then identify a polynomial solvable class.
In Section 4, we propose a new approximation algorithm for ${\rm (P_{ball})}$ and establish the approximation bound. Section 5 presents some numerical results. Conclusions are made in Section 6.

Throughout the paper, the notation ``:=''denotes ``define''. We denote by $\Bbb R^n$ the $n$-dimensional vector space. 
For two matrices $A,B\in S^n$, let $A\succ(\succeq)B$ denote that $A-B$
 is positive (semi)definite. $\|A\|:=\max_{\|x\|=1}\|Ax\|$ is the induced norm of the matrix $A$. ${\rm Diag}(a_1,\ldots,a_n)$ stands for the diagonal matrix with principal diagonal elements $a_1,\ldots,a_n$. The one-dimensional intervals
 $\{x \mid a<x<b\}$, $\{x  \mid a<x \le b\}$, $\{x \mid a\le x<b\}$ and $\{x \mid a\le x \le b\}$ are denoted by $(a,b)$, $(a,b]$, $[a,b)$ and $[a,b]$, respectively.
Denote the left- and right-hand limits by $\lim_{x\rightarrow a^-}f(x)$ and $\lim_{x\rightarrow a^+}f(x)$, respectively.  For $a\in\Bbb R^1$, $[a]$ denotes the largest integer less than or equal to $a$. For a nonnegative integer $n$, define the double factorial as $n!!=\prod_{k=0}^{[n/2]}(n-2k)$ and $0!!=1$.
${\rm Pr(\cdot)}$ stands for the probability.

\section{A new convex relaxation}
We propose a new convex relaxation for  ${\rm (P_{\chi})}$, denoted by $({\rm CR_\chi})$, and then show that $v({\rm CR_\chi})=v({\rm SDP_\chi})=v({\rm SOCP_\chi})$ when either $\chi=[-1,1]^n$ or $\chi=\{x\mid \|x\|\leq1\}$.
For convenience, throughout this paper, when $\chi=[-1,1]^n$, we re-denote
$({\rm CR_\chi})$, $({\rm SDP_{\chi}})$ (\ref{sdp}) and $({\rm SOCP_{\chi}})$ (\ref{socp}) by $({\rm CR_{box}})$, $({\rm SDP_{box}})$ and $({\rm SOCP_{box}})$, respectively. Similarly, when $\chi= \{x\mid \|x\|\leq1\}$,
$({\rm CR_\chi})$, $({\rm SDP_{\chi}})$ and $({\rm SOCP_{\chi}})$ are re-denoted by $({\rm CR_{ball}})$, $({\rm SDP_{ball}})$ and $({\rm SOCP_{ball}})$, respectively.


We first reformulate ${\rm (P_{\chi})}$ as an equivalent smooth optimization problem (see also Proposition 2.1 and (11) in \cite{HA13}):
\begin{eqnarray}
{\rm (P_{\chi})'}~~&\max_{x,~\zeta}&~ \zeta\nonumber\\
&{\rm s.~t.}& \omega_i(\|x\|^2-2(x^i)^Tx+
\|x^i\|^2)\geq \zeta,~i=1,\ldots,m, \label{mu:1}\\
&&x\in \chi.\nonumber
\end{eqnarray}
Relaxing the nonlinear term $\|x\|^2$ in (\ref{mu:1}) to a constant upper bound
\begin{equation}
\mu:=\max_{x\in \chi}\|x\|^2~({\rm or~an~upper~bound~if~it~is~hard~to~solve}) \label{mu:2}
\end{equation}
 yields the following convex relaxation:
\begin{eqnarray*}
{\rm (CR_{\chi})}~~&\max_{x,~\zeta}&~ \zeta\\
&{\rm s.~t.}& \omega_i(\mu-2(x^i)^Tx+
\|x^i\|^2)\geq \zeta,~i=1,\ldots,m,\\
&&x\in \chi.
\end{eqnarray*}
When $\chi= [-1,1]^n$, according to the definition (\ref{mu:2}), we have $\mu=n$. Then, the above convex relaxation ${\rm (CR_{\chi})}$ becomes
the following $\emph{linear programming}$ problem:
\begin{eqnarray*}
{\rm (CR_{box})}~~&\max_{x,~\zeta}&~ \zeta\\
&{\rm s.~t.}& \omega_i(n-2(x^i)^Tx+\|x^i\|^2)\geq \zeta,~i=1,\ldots,m,\\
&&x\in [-1,1]^n.
\end{eqnarray*}
Surprisingly,
it is actually as tight as both the SDP and the SOCP relaxations.
\begin{thm}\label{thm00}
$v({\rm CR_{box}})=v({\rm SDP_{box}})=v({\rm SOCP_{box}}).$
\end{thm}
\begin{proof} The proof is divided into  two parts.

(a) First we construct a feasible solution of $({\rm SDP_{box}})$ from an optimal solution of $({\rm CR_{box}})$.
Since the feasible region is compact, $({\rm CR_{box}})$ has an optimal solution, denoted by $(x^*,\zeta^*)$.
It is trivial to see that $v({\rm P_{box}})>0$. Thus, we have $\zeta^*=v({\rm CR_{box}})\ge v({\rm P_{box}})>0$. Let $x_i^*$ be the $i$-th component of $x^*$ for $i=1,\ldots,n$ and then define
\begin{eqnarray}
\widehat{Z}:=&\frac{1}{\zeta^*}\left(\left(\begin{array}{c} x^*\\ 1\end{array}\right)
\left(\begin{array}{cc} {x^*}^T&1\end{array}\right)+{\rm Diag}\left(1-(x_1^*)^2, \ldots,1-(x_n^*)^2,0\right)\right).\label{zz}
\end{eqnarray}
We can verify that
$\widehat{Z}\succeq 0$ and
\begin{eqnarray*}
&&\widehat{Z}_{jj}= \widehat{Z}_{n+1,n+1},~j=1,\ldots,n,\\
&&\omega_i{\rm Tr}(A^i\widehat{Z})=\frac{\omega_i}{\zeta^*}(n-2(x^i)^Tx^*+\|x^i\|^2)\geq1,~i=1,\ldots,m.
\end{eqnarray*}
That is,
$\widehat{Z}$ is a feasible solution of $({\rm SDP_{box}})$ (\ref{sdp}).  Then, it holds that
 \[
 \frac{1}{v({\rm SDP_{box}})}\leq \widehat{Z}_{n+1,n+1}=\frac{1}{\zeta^*},
 \]
or equivalently,
\begin{equation}
 v({\rm SDP_{box}})\geq v({\rm CR_{box}})=\zeta^*.\label{pf1:1}
 \end{equation}

(b) Next we construct a feasible solution of $({\rm CR_{box}})$ from an optimal solution of $({\rm SOCP_{box}})$. let $Z^*=\left(\begin{array}{cc}\widetilde{Z}^* & z^*\\{z^*}^T & Z^*_{n+1,n+1}\end{array}\right)$ be an optimal solution of $({\rm SOCP_{box}})$,
where $z^*=(Z^*_{1,n+1},\ldots,Z^*_{n,n+1})^T\in\Bbb R^n$.
According to (\ref{socp:3}), we have
\begin{equation}
(Z^*_{j,n+1})^2 \le Z^*_{jj}Z^*_{n+1,n+1},~j=1,\ldots,n.\label{pf1:2}
\end{equation}
It follows from (\ref{socp:2}) (where the set $\chi$ is defined in (\ref{k:box})) that
\begin{equation}
 Z^*_{jj}  \le  Z^*_{n+1,n+1},~j=1,\ldots,n,\label{pf1:3}
\end{equation}
which further implies that
\begin{equation}
 {\rm Tr}(\widetilde{Z}^*)=\sum_{j=1}^nZ_{jj}^*\le
\sum_{j=1}^nZ_{n+1,n+1}^*=nZ_{n+1,n+1}^*.
\label{pf1:4}
\end{equation}
Since $v({\rm SOCP_{box}})\ge v({\rm P_{box}})>0$, we have $Z^*_{n+1,n+1}=\frac{1}{v({\rm SOCP_{box}})}>0$. Define
$\widetilde{x}:= \frac{z^*}{Z^*_{n+1,n+1}}$.
Then, we have
\[
\widetilde{x}_j^2=\frac{(Z^*_{j,n+1})^2}{(Z^*_{n+1,n+1})^2}
\le \frac{ Z^*_{jj}Z^*_{n+1,n+1} }{(Z^*_{n+1,n+1})^2}=
 \frac{ Z^*_{jj} }{Z^*_{n+1,n+1}}\le1,
~j=1,\ldots,n,
\]
where the two inequalities follows from (\ref{pf1:2}) and (\ref{pf1:3}), respectively.  Using (\ref{pf1:4}) and (\ref{socp:1}), we obtain
\[
\omega_i(n-2(x^i)^T\widetilde{x}+\|x^i\|^2)\geq \frac{\omega_i({\rm Tr}(\widetilde{Z}^*)-2(x^i)^Tz^*+\|x^i\|^2Z^*_{n+1,n+1})}{Z^*_{n+1,n+1}}\geq
\frac{1}{Z^*_{n+1,n+1}}.
\]
Then,
$\left(\widetilde{x},\frac{1}{Z^*_{n+1,n+1}}\right)$
is a feasible solution for $({\rm CR_{box}})$.
Consequently, we obtain
 \begin{equation}
v({\rm CR_{box}})\geq \frac{1}{Z_{n+1,n+1}^*}=v({\rm SOCP_{box}}).\label{pf1:5}
 \end{equation}

Combining (\ref{pf1:1}), (\ref{pf1:5}) and the trivial fact $v({\rm SOCP_{box}})\geq v({\rm SDP_{box}})$ \cite{HA13}, we have
\[
v({\rm SOCP_{box}})\geq v({\rm SDP_{box}})\geq v({\rm CR_{box}}) \geq v({\rm SOCP_{box}}),
\]
which completes the proof.
\end{proof}
\begin{rem}
To our knowledge, the result $v({\rm SDP_{box}})=v({\rm SOCP_{box}})$ is also new. Though
it can be observed from the numerical results in \cite{HA13}, the authors there failed to realize that this is always true.
\end{rem}

Below we study the ball case $\chi= \{x\mid \|x\|\leq1\}$. According to the definition (\ref{mu:2}), we have $\mu=1$. The new convex relaxation ${\rm (CR_{\chi})}$ is recast as the following $\emph{second-order cone programming}$ problem:
\begin{eqnarray*}
{\rm (CR_{ball})}~~&\max_{x,~\zeta}&~ \zeta\\
&{\rm s.~t.}& \omega_i(1-2(x^i)^Tx+\|x^i\|^2)\geq \zeta,~i=1,\ldots,m,\\
&& \|x\|\leq 1.
\end{eqnarray*}
It should be noted that the size of ${\rm (CR_{ball})}$ is much smaller than that of
$({\rm SDP_{ball}})$ or $({\rm SOCP_{ball}})$.
For the tightness of ${\rm (CR_{ball})}$,  similar to Theorem \ref{thm00}, we have
the following result.
\begin{thm}\label{thm01}
$v({\rm CR_{ball}})=v({\rm SDP_{ball}})=v({\rm SOCP_{ball}}).$
\end{thm}
\begin{proof}
 Similar to that of Theorem \ref{thm00}, the proof contains two constructive parts.

(a) First we construct a feasible solution of $({\rm SDP_{ball}})$ from an optimal solution of $({\rm CR_{ball}})$.
Since the feasible region of $({\rm CR_{ball}})$ is compact, $({\rm CR_{ball}})$ has an optimal solution, denoted by $(x^*,\zeta^*)$. It is trivial to see that $v({\rm P_{ball}})>0$. Thus, we have $\zeta^*=v({\rm CR_{ball}})\ge v({\rm P_{ball}})>0$.
Define
\begin{eqnarray}
\widehat{Z}:=\frac{1}{\zeta^*}\left(\left(\begin{array}{c} x^*\\ 1\end{array}\right)
\left(\begin{array}{cc} {x^*}^T&1\end{array}\right)+
\left(\begin{array}{cc} \frac{1-\|x^*\|^2}{n}\cdot I&0\\0& 0\end{array}\right)
\right).\label{yy}
\end{eqnarray}
We can show that $\widehat{Z}\succeq 0$ and
\begin{eqnarray*}
&& \omega_i{\rm Tr}(A^i\widehat{Z})= \frac{1}{\zeta^*}(1-2(x^i)^Tx^*+\|x^i\|^2)\geq1,~i=1,\ldots,m,\\
&& \widehat{Z}_{11}+\widehat{Z}_{22}+\ldots+\widehat{Z}_{nn}=\frac{1}{\zeta^*}= \widehat{Z}_{n+1,n+1}.
\end{eqnarray*}
That is, $\widehat{Z}$
is a feasible solution of $({\rm SDP_{ball}})$ (\ref{sdp}), where the set $\mathcal{K}$ is defined in (\ref{k:ball}).
It follows that
\begin{equation}
 v({\rm SDP_{ball}})\geq v({\rm CR_{ball}})=\zeta^*.\label{pf2:1}
 \end{equation}

(b) Next we construct a feasible solution of $({\rm CR_{ball}})$ from an optimal solution of $({\rm SOCP_{ball}})$.
 Let $Z^*=\left(\begin{array}{cc}\widetilde{Z}^* & z^*\\{z^*}^T & Z^*_{n+1,n+1}\end{array}\right)$ be an optimal solution of $({\rm SOCP_{ball}})$,
where $z^*=(Z^*_{1,n+1},\ldots,Z^*_{n,n+1})^T$.
Then, (\ref{pf1:2}) holds. Moreover,
it follows from (\ref{socp:2}) (where the set $\mathcal{K}$ is defined in (\ref{k:ball})) that
\begin{equation}
 {\rm Tr}(\widetilde{Z}^*)=\sum_{j=1}^nZ_{jj}^*\le
 Z_{n+1,n+1}^*.\label{pf2:3}
\end{equation}
Since $v({\rm SOCP_{ball}})\ge v({\rm P_{ball}})>0$, we have $Z^*_{n+1,n+1}=\frac{1}{v({\rm SOCP_{ball}})}>0$. Define
$\widetilde{x} := \frac{z^*}{Z^*_{n+1,n+1}}$.
Using (\ref{pf1:2}) and (\ref{pf2:3}), we can show that
\[
\|\widetilde{x}\|^2=\frac{\sum_{j=1}^n(Z^*_{j,n+1})^2}{(Z^*_{n+1,n+1})^2}
\le \frac{\sum_{j=1}^n Z^*_{jj}Z^*_{n+1,n+1} }{(Z^*_{n+1,n+1})^2}=
 \frac{ \sum_{j=1}^nZ^*_{jj} }{Z^*_{n+1,n+1}}\le1.
\]
Furthermore,  for $i=1,\ldots,m$, it follows from (\ref{pf2:3}) and (\ref{socp:1})  that
\[
\omega_i(1-2(x^i)^T\widetilde{x}+\|x^i\|^2)\geq \frac{\omega_i({\rm Tr}(\widetilde{Z}^*)-2(x^i)^Tz^*+\|x^i\|^2Z^*_{n+1,n+1})}{Z^*_{n+1,n+1}}\geq
\frac{1}{Z^*_{n+1,n+1}}.
\]
Thus,
$\left(\widetilde{x},\frac{1}{Z^*_{n+1,n+1}}\right)$
is a feasible solution for $({\rm CR_{ball}})$.
Consequently, we obtain
 \begin{equation}
v({\rm CR_{ball}})\geq \frac{1}{Z_{n+1,n+1}^*}=v({\rm SOCP_{ball}}).\label{pf2:5}
 \end{equation}

Combining (\ref{pf2:1}), (\ref{pf2:5}) and the fact $v({\rm SOCP_{ball}})\geq v({\rm SDP_{ball}})$ \cite{HA13} yields
\[
v({\rm SOCP_{ball}})\geq v({\rm SDP_{ball}})\geq v({\rm CR_{ball}}) \geq v({\rm SOCP_{ball}}).
\]
The proof is complete.
\end{proof}
\begin{rem}
The reason behind the equivalence as shown in
Theorems \ref{thm00} and \ref{thm01} may be that the constant $\mu$ defined in (\ref{mu:2}) is the smallest concave function majorizing $\|x\|^2$ over either $[-1,1]^n$ or $\{x\mid \|x\|\leq1\}$.
\end{rem}

\section{Computational complexity}
In this section, we first identify a polynomially solvable class of $({\rm P_{ball}})$ and then
prove that $({\rm P_{ball}})$ is generally NP-hard.
\subsection{Polynomial solvable cases}
In a very recent paper \cite{WX}, we have shown that $({\rm P_{ball}})$ enjoys a hidden convexity when $m\le n-1$. This sufficient condition is further improved as follows.
\begin{thm}\label{poly}
Suppose the linear system
\begin{eqnarray}
(x^i)^Tx\le0,~i=1,\ldots,m\label{sufc}
\end{eqnarray}
has a nonzero solution.
Then $v({\rm P_{ball}})=v({\rm CR_{ball}})$ and $({\rm P_{ball}})$
can be solved in polynomial time. Moreover, the above sufficient condition is satisfied when
$
m\le n.
$
\end{thm}
\begin{proof}
Since $({\rm CR_{ball}})$ is a second-order cone programming problem, we can solve it in polynomial time and obtain an optimal solution
$(x^*,\zeta^*)$.
For each $j\in\{1,\ldots,n\}$, solve the following absolute-value linear programming problem (or equivalently, two linear programming problems):
\[
v^*_j:=\max_{(\ref{sufc})} |x_j|=\max\left\{\max_{ x_j\ge0~{\rm and}~(\ref{sufc})} x_j,~\max_{ x_j\le0~{\rm and}~(\ref{sufc})} -x_j\right\}.
\]
If $v^*_j=0$ for $j=1,\ldots,n$, the linear system (\ref{sufc}) has no nonzero solution. Otherwise, there is an index $k$ such that $v^*_k>0$. If $v^*_k<+\infty$, let  $\widetilde{x}$ be an optimal solution of $\max_{(\ref{sufc})} |x_k|$. Otherwise, let $\widetilde{x}$ be any nonzero feasible solution of (\ref{sufc}).
Then, $\|\widetilde{x}\|>0$.
Define
\[
\alpha=\frac{-(x^*)^T\widetilde{x}+\sqrt{\|\widetilde{x}\|^2(1-\|x^*\|^2)+
((x^*)^T\widetilde{x})^2}}{\|\widetilde{x}\|^2}.
\]
Clearly, $\alpha\ge0$ and $\|x^*+\alpha \widetilde{x}\|=1$.
Then, we have
\begin{eqnarray*}
f(x^*+\alpha \widetilde{x})
&=&\min_{i=1,\ldots,m}\omega_i   \|x^*+\alpha \widetilde{x}-x^i\|^2 \\
&=&\min_{i=1,\ldots,m}\omega_i \left(\|x^*+\alpha \widetilde{x}\|^2-
2(x^i)^T(x^*+\alpha \widetilde{x})+\|x^i\|^2\right)\\
&=&\min_{i=1,\ldots,m}\omega_i \left(1-
2(x^i)^T(x^*+\alpha \widetilde{x})+\|x^i\|^2\right)\\
&\ge&
\min_{i=1,\ldots,m}\omega_i\left(1-2(x^i)^T(x^*)+\|x^i\|^2\right)\\
&=& v({\rm CR_{ball}})\ge v({\rm P_{ball}}),
\end{eqnarray*}
where the first inequality follows from the definition of $\widetilde{x}$ and the nonnegativity of $\alpha$.
Therefore, $x^*+\alpha \widetilde{x}$ is an optimal solution of $({\rm P_{ball}})$.

Next, we show that the  linear system (\ref{sufc}) has a nonzero solution under the assumption $m\leq n$. 
Since $m-1< n$, the following linear equations
\[
(x^i)^Tx=0,~i=1,\ldots,m-1,
\]
have an infinite number of solutions. Let $\widehat{x}$ be a nonzero solution. Define
\[
\widetilde{x}:=\left\{\begin{array}{ll}
-\left((x^m)^T\widehat{x}\right)\widehat{x},&{\rm if}~(x^m)^T\widehat{x} \neq 0,\\
\widehat{x},&{\rm otherwise.} \end{array}\right.
\]
Then, we can verify that $\widetilde{x}$ is a nonzero solution of (\ref{sufc}).
\end{proof}

As shown in
the following example, our new sufficient condition could not be further improved.
\begin{exam}
Let $n=1,~m=2$ and consider the following instance of ${\rm (P_{ball})}$: 
\begin{eqnarray*}
 &\max_{x^2\leq 1,~x\in\Bbb R^1}& \min\{(x-1)^2,~(x+1)^2\}.
\end{eqnarray*}
Correspondingly, ${\rm(CR_{ball})}$ is rewritten as
$ \max_{x^2\leq 1,~x\in\Bbb R^1}  ~\min\{2-2x,~2+2x\}$.
Our sufficient condition fails to hold, since the corresponding linear system (\ref{sufc})
has no nonzero solution.
It is not difficult to verify  that
\[
v{\rm (P_{ball})}=1 < v({\rm CR_{ball}})=2.
\]
\end{exam}

\subsection{NP-hardness}
In general, we show that the problem ${\rm (P_{ball})}$ remains NP-hard even when $\omega_1=\omega_2=\ldots=\omega_m$. Firstly, we present a useful lemma.

\begin{lem} \label{lem4}
Assume $Q\succ0$. Consider the binary quadratic program
\begin{eqnarray}
{\rm (BQP)}~~&\max& ~x^TQx  \label{bqp:1}\\
&{\rm s.~t.}& x\in\{-1,1\}^n, \nonumber
\end{eqnarray}
and the quadratic constrained quadratic program
\begin{eqnarray}
{\rm (QCQP)}~~&\max_{x,~s}& ~x^TQx+s  \label{qcqp:1}\\
&{\rm s.~t.}& x\in [s-1,1-s]^n, \nonumber\\
&&x^TQx\le 1.\nonumber
\end{eqnarray}
Then, we have
\begin{eqnarray*}
v({\rm QCQP})=\left\{\begin{array}{cl}
2-\frac{1}{\sqrt{v({\rm BQP})}},& v({\rm BQP})\ge 1,\\
1,& v({\rm BQP})< 1.
\end{array}
\right.  
\end{eqnarray*}
\end{lem}
\begin{proof}
Since $Q\succ0$, ${\rm (BQP)}$ is equivalent to the following continuous relaxation:
\begin{eqnarray*}
{\rm (QP)}~~&\max& ~x^TQx \\
&{\rm s.~t.}& x\in [-1,1]^n,
\end{eqnarray*}
in the sense that $v({\rm BQP})=v({\rm QP})$. Besides, we have $v({\rm BQP})>0$.

According to the definition of (QCQP), it follows from $s-1\le 1-s$ that  $s\le 1$. Moreover, it is trivial to verify that
\[
\max_{x\in [s-1,1-s]^n} x^TQx =(1-s)^2\cdot v({\rm QP})=(1-s)^2\cdot v({\rm BQP}).
\]
Therefore, we have
\begin{eqnarray}
&&v({\rm QCQP})\nonumber\\
&=&\max_{s\le1}\left\{ \min\{(1-s)^2\cdot v({\rm BQP}),~1\}+s\right\}\nonumber\\
&=&\max\left\{
\max_{1-\frac{1}{\sqrt{v({\rm BQP})}}\le s\le1} \left\{ (1-s)^2\cdot v({\rm BQP})+s\right\},~
\max_{1-\frac{1}{\sqrt{v({\rm BQP})}}\ge s } \left\{  1+s\right\}
\right\}\nonumber\\
&=&\max\left\{ \max\left\{1+1-\frac{1}{\sqrt{v({\rm BQP})}},1\right\},~
  1+1-\frac{1}{\sqrt{v({\rm BQP})}}
\right\}\label{np:3}\\
&=& \left\{\begin{array}{cl}
2-\frac{1}{\sqrt{v({\rm BQP})}},& v({\rm BQP})\ge 1,\\
1,& v({\rm BQP})< 1,
\end{array}
\right. \nonumber
\end{eqnarray}
where Equation (\ref{np:3}) holds since the strictly convex function $(1-s)^2\cdot v({\rm BQP})+s$ attains its maximum at a vertex.
\end{proof}

 \begin{thm} \label{thm32}
The ball-constrained problem ${\rm (P_{ball})}$ is NP-hard.
\end{thm}
\begin{proof}
For readability, we present here the sketch of the proof and leave the tedious supplementary proof in Appendix A.

Given $a=(a_1,\ldots,a_n)^T$ with integer entries,  the partition problem (PP) is to ask whether the following equation
\begin{equation}
a^Tx=0,~x\in\{-1,1\}^n \label{pp}
\end{equation}
has a solution. The NP-hardness of (PP) can be found in \cite{GJ}.  Without loss of generality, we assume $a_i\neq0$ for $i=1,\ldots,n$.

For any $t ~(0<t<1)$, define
\begin{eqnarray}
\beta(t)&=& \frac{ 1-\sqrt{1-t}}{t\sqrt{1-t}},\\
\gamma(t)&=& 2\beta(t)+t\beta(t)^2.
\end{eqnarray}
The following equations
\[
\left\{\begin{array}{rcl}
\Lambda_{ii}(t)^{-1}+a_i^2\gamma(t)\Lambda_{ii}(t)^{-2} &=&1,~i=1,\ldots,n\\
\sum_{i=1}^n a_i^2\Lambda_{ii}(t)^{-1}&=&t
\end{array}\right.
\]
are equivalent to
\[
\left\{\begin{array}{rcl}
\Lambda_{ii}(t)&=&\frac{1}{2}+\frac{1}{2}
\sqrt{1+4a_i^2\gamma(t)},~i=1,\ldots,n,\\
g(t)&=&0,
\end{array}\right.
\]
where
\[
g(t)=t-\sum_{i=1}^n \frac{2a_i^2}{1+\sqrt{1+4a_i^2 \left(2\frac{ 1-\sqrt{1-t}}{t\sqrt{1-t}} +\left(\frac{ 1-\sqrt{1-t}}{t\sqrt{1-t}}\right)^2t\right)}}.
\]
Since $
\lim_{t\rightarrow 0^+}\frac{ 1-\sqrt{1-t}}{t\sqrt{1-t}}=\frac{1}{2}$ and
$\lim_{t\rightarrow 1^-}\frac{ 1-\sqrt{1-t}}{t\sqrt{1-t}}=+\infty$,
  we have
\[
\lim_{t\rightarrow 0^+}g(t)=-\frac{1}{2}
\sum_{i=1}^n \left(\sqrt{1+4 a_i^2}-1\right)<0,~
\lim_{t\rightarrow 1^-}g(t)=1>0.
\]
Therefore, there is a $t^*\in(0,1)$ such that $g(t^*)=0$. Notice that the root $t^*$ can be approximately obtained in polynomial time by using the bisection method.
 For readability, in the following proof, we use the exact value $t^*$. A supplementary proof based on an approximated $t^*$ is presented in Appendix A. Define
 \begin{eqnarray}
 \Lambda&:=&{\rm Diag}(\Lambda_{11}(t^*),\ldots,\Lambda_{nn}(t^*)),\label{Lmbd}\\
L&:=&\Lambda^{-1/2}\left(I+\beta(t^*) \Lambda^{-1/2}aa^T\Lambda^{-1/2}\right),
\label{LL}
\end{eqnarray}
and $L_i$ be the $i$-th row of $L$ for $i=1,\ldots,n$. Then, we can verify that
\[
\|L_i\|^2=\Lambda_{ii}(t^*)^{-1}+a_i^2\gamma(t^*)\Lambda_{ii}(t^*)^{-2}=1,
~i=1,\ldots,n.
\]
Moreover, we have
\[
LL^T=\Lambda^{-1}+\frac{\Lambda^{-1}aa^T\Lambda^{-1}}
{1-a^T\Lambda^{-1}a}=(\Lambda-aa^T)^{-1},
\]
where the last equality follows from the Sherman-Morrison formula \cite{Hag}.

Denote by ($\rm  BQP^*$) and ($\rm QCQP^*$) the problems obtained by replacing the Hessiam matrix $Q$ with  $\frac{1}{4}\left( \Lambda-aa^T\right)$ in (BQP) (\ref{bqp:1}) and (QCQP) (\ref{qcqp:1}), respectively. Without loss of generality, we assume $n\ge 5$. Then, we have
\[
\frac{1}{4}{\rm Tr}(\Lambda)> \frac{n}{4}>1.
\]
Therefore, (PP) (\ref{pp}) has a solution if and only if $v{\rm(BQP^*)}= \frac{1}{4}{\rm Tr}(\Lambda)(>1)$.
According to Lemma \ref{lem4}, it is equivalent to $v{\rm(QCQP^*)}=2-\frac{2}{\sqrt{{\rm Tr}(\Lambda)}}(>1)$. Thus,
($\rm QCQP^*$) remains NP-hard.

By introducing $y=\frac{1}{2}L^{-1}x$, we can reformulate
(${\rm QCQP^*}$) as
\begin{eqnarray*}
 &\max_{y,~s}& ~y^Ty+s \\
&{\rm s.~t.}& 2L_iy\le 1-s,~i=1,\ldots,n,\\
& & -2L_iy\le 1-s,~i=1,\ldots,n,\\
&&y^Ty\le 1.\nonumber
\end{eqnarray*}
It is further equivalent to
\begin{eqnarray*}
\max_{y^Ty\le 1} \left\{ \min_{i=1,\ldots,n} \left\{ \|y-L_i^T\|^2,~ \|y+L_i^T\|^2 \right\}\right\},
\end{eqnarray*}
 which corresponds to
 a special case of ${\rm (P_{ball})}$ with $m=2n$.
Therefore,  ${\rm (P_{ball})}$ is NP-hard.
\end{proof}

\section{Approximation algorithms}
In this section, we propose a new simple approximation algorithm for solving ${\rm (P_{ball})}$ and then establish the approximation bound of $\frac{1-O(\sqrt{\ln(m)/n})}{2}$.

\subsection{Approximation algorithms for (${\rm P_{\chi}}$)}\label{sec}
We show that the following general approximation algorithm for $({\rm P_{\chi}})$ proposed in \cite{HA13} should be slightly fixed
and then it can be significantly simplified when $\chi=[-1,1]^n$.

~

\begin{center}
\fbox{\shortstack[l]{
{\bf An approximation algorithm for $({\rm P_{\chi}})$}\\
1.~ Input $\rho\in(0,1)$ and $x^i$ for $i=1,\ldots,m$. Let $\alpha=\sqrt{2\ln(m/\rho)}$. \\
2.~
 Solve $({\rm SDP_{\chi}})$ (or $({\rm SOCP_{\chi}})$)  and return the optimal solution $Z^*\in S^{n+1}$.\\~~~~~Set $b^i=(\sqrt{Z^*_{11}}x^i_1,\ldots,\sqrt{Z^*_{nn}}x^i_n)^T$  for $i=1,\ldots,m$.\\
3.~    Repeatedly generate $\xi=(\xi_1,\ldots,\xi_n)^T$ with independent $\xi_i$ taking the \\~~~~~value $\pm1$ with equal probability until  $(b^i)^T\xi<\alpha\|b^i\|$ for $i=1,\ldots,m$.\\
4.~ Output $
\widetilde{x}=\left(\frac{\sqrt{{Z^*_{11}}}\xi_1}{\sqrt{{Z^*_{n+1,n+1}}}},
\frac{\sqrt{{Z^*_{22}}}\xi_2}{\sqrt{{Z^*_{n+1,n+1}}}},
\ldots,\frac{\sqrt{{Z^*_{nn}}}\xi_n}{\sqrt{{Z^*_{n+1,n+1}}}}\right)^T$.
}}
\end{center}

~

It is trivial to see that the above algorithm breaks down when there is an index $k\in\{1,\ldots,m\}$ such that $\|x^k\|=0$.
Fortunately, Step 3 of the above algorithm can be revised to generate $\xi$  by repeated random sampling such that
\[
(b^i)^T\xi<\alpha\|b^i\|, ~{\rm for}~ i\in I:=\{i\in\{1,\ldots,m\}\mid~\|x^i\|>0\}.
\]
The existence of $\xi$ is guaranteed by the following inequality:
\[
{\rm Pr}\left((b^i)^T\xi<\alpha\|b^i\|,~i\in I \right)\ge 1-\rho>0,
\]
which follows from the following well-known result:
\begin{thm} {\rm \cite[Lemma A.3]{Ben2002}}\label{thm0}
Let $\xi\in \{-1,1\}^n$ be a random vector, componentwise independent, with
\[
{\rm Pr}(\xi_j=1)={\rm Pr}(\xi_j=-1)=\frac{1}{2},~\forall j=1,\ldots,n.
\]
Let $b\in \Bbb R^n$ and $\|b\|> 0$. Then for any $\alpha>0$,
\[
{\rm Pr}(b^T\xi\geq \alpha\|b\|)\leq e^{-\alpha^2/2}.
\]
\end{thm}
It is not difficult to verify that the approximation bound established in \cite{HA13} is still satisfied for the fixed algorithm.
\begin{thm} {\rm\cite[Theorem 3]{HA13}}\label{thm03}
Let $\gamma_1^*$ be defined in (\ref{gamma:0}).
Then, for the solution $\widetilde{x}$ returned by the above (fixed) algorithm,  we have
\[
v({\rm SDP_{\chi}})\geq v({\rm P_{\chi}})\ge f(\widetilde{x})\geq \frac{1-\alpha\sqrt{\gamma_1^*}}{2}v({\rm SDP_{\chi}})\geq \frac{1-\alpha\sqrt{\gamma_1^*}}{2}v({\rm P_{\chi}}).
\]
\end{thm}


Now, we focus on the special case $({\rm P_{\rm box}})$. According to the construction of the optimal  solution of $({\rm SDP_{\rm box}})$ defined in (\ref{zz}),  we have
\begin{equation}
Z_{11}^*=Z_{22}^*=\ldots=Z_{nn}^*=Z_{n+1,n+1}^*. \label{Z:eq}
\end{equation}
It turns out that  $\gamma_1^*$
defined in (\ref{gamma:0})  can be simplified to $\frac{1}{n}$. The solution structure (\ref{Z:eq}) has also been shown in Proposition 2.10 in \cite{HA13}. However, the authors there seem not to realize that
the equalities (\ref{Z:eq}) can be used to  simplify the above approximation algorithm for $({\rm P_{box}})$. Actually, as seeing in the following, we do not have to solve any convex relaxation for $({\rm P_{box}})$.

~

\begin{center}
\fbox{\shortstack[l]{
{\bf A simplified approximation algorithm for $({\rm P_{box}})$}\\
1.~ Input $\rho\in(0,1)$ and $x^i$  for $i=1,\ldots,m$. Let $\alpha=\sqrt{2\ln(m/\rho)}$. \\
 2.~ Repeatedly generate $\xi\in\Bbb R^n$ with independent $\xi_i$ taking the value $\pm1$ with \\~~~~~equal probability until  $(x^i)^T\xi<\alpha\|x^i\|$ for $i\in\{1,\ldots,m\}\setminus\{k\mid~\|x^k\|=0\}$.\\
3.~ Output $\widetilde{x}=\xi.$
}}
\end{center}

~

Unfortunately, the above simplified approach for $({\rm P_{box}})$ can not be similarly extended to $({\rm P_{ball}})$, since the equalities (\ref{Z:eq}) no longer hold true for the ball-constrained case as shown in the following example.
\begin{exam}
Let $n=2,~m=3$, $x^1=\left(\begin{array}{c}1\\2\end{array}\right)$, $x^2=\left(\begin{array}{c}2\\3\end{array}\right)$,
$x^3=\left(\begin{array}{c}1\\5\end{array}\right)$, and $\omega_1=\omega_2=\omega_3=1$. Solving the SDP relaxation ${\rm (SDP_{ball})}$ (\ref{sdp}) by SDPT3 within CVX \cite{GrB} yields the optimal value $\frac{1}{v{\rm (SDP_{ball})}}\approx0.0955$ and an optimal solution
\[
Z^*=\left(\begin{array}{ccc}
 0.0191 &   0.0382&   -0.0427\\
    0.0382&    0.0764&   -0.0854\\
   -0.0427 &  -0.0854 &   0.0955
\end{array}\right).
\]
Then, we have $Z_{11}^*<Z_{22}^*<Z_{33}^*$. Moreover, since solving a modified model of ${\rm (SDP_{ball})}$ with an additional constraint $Z_{11}=Z_{22}$ yields
a larger optimal objective value $0.0976$,  
the equalities (\ref{Z:eq}) can not hold for  any  optimal solution of ${\rm (SDP_{ball})}$.

Besides, according to the definition (\ref{gamma:0}), we have
$
\gamma_1^*=\frac{0.0764}{ 0.0191+0.0764}=0.8$.
It follows that  the approximation bound established in Theorem \ref{thm03} is quite poor for this example as it is negative:
\[
\frac{1-\alpha\sqrt{\gamma_1^*}}{2}=\frac{1-\sqrt{1.6\ln(3/\rho)}}{2}
\le\frac{1-\sqrt{1.6\ln(3)}}{2}<-0.1629<0.
\]
\end{exam}

\subsection{{Tail estimation of uniformly sampling on the sphere}}
Similar to Theorem \ref{thm0}, we establish a tail estimation theory for uniformly sampling over the sphere of radius $\sqrt{n}$.
It will serve  as a key lemma in the analysis of our new approximation algorithm presented in the next subsection.

\begin{thm}\label{lem1}
Let $b=\{b_1,\ldots,b_n\}\in \Bbb R^n$ and $\|b\|>0$.
Let $\eta=(\eta_1,\ldots,\eta_n)^T$ be
uniformly
distributed over the sphere of radius $\sqrt{n}$ in $\Bbb R^n$ (that is, $\|\eta\|=\sqrt{n}$). Then, for any $\alpha>0$ and any $n\ge 2$, we have
\begin{equation}
{\rm Pr}\left( b^T\eta \geq \alpha\|b\| \right)=S(n,\alpha):=\left\{\begin{array}{cl}\frac{
\int_{ \alpha/{\sqrt{n}}}^{1}\left(\sqrt{1-t^2}\right)^{n-3}{\rm d}t}
{2\int_{0}^{1}\left(\sqrt{1-t^2}\right)^{n-3}{\rm d}t},&{\rm ~if~}\alpha\le\sqrt{n},\\
0,&{\rm ~if~}\alpha>\sqrt{n}.\end{array}\right. \label{p2:1}
\end{equation}
Moreover, $S(n,\alpha)$
is strictly decreasing in terms of $\alpha\in(0,\sqrt{n})$ and satisfies that
\begin{equation}
S(n,\alpha)<e^{-0.45\alpha^2},~\forall n\ge 2,~\forall \alpha>0.\label{ie:1}
\end{equation}
\end{thm}
\begin{proof}
According to the Cauchy-Schwartz inequality, we have
$b^T\eta \le \|b\|\cdot\|\eta\|=\sqrt{n}\|b\|$,
which implies that
\[
0\le{\rm Pr}\left( b^T\eta \geq \alpha\|b\| \right) \le {\rm Pr}\left( b^T\eta> \sqrt{n}\|b\| \right)=0,
\]
when $\alpha>\sqrt{n}$ and $\|b\|>0$. Below we are sufficient to assume $\alpha\le \sqrt{n}$.

When $n=2$, define $y=(\eta_1/\sqrt{2},\eta_2/\sqrt{2})^T$, then $\|y\|=1$ and
\[
{\rm Pr}\left( b^T\eta \geq \alpha\|b\| \right)
={\rm Pr}\left( \left(\frac{b}{\|b\|}\right)^Ty \geq \frac{\alpha}{\sqrt{2}} \right)
=\frac{\arccos\frac{\alpha}{\sqrt{2}}}{\pi}.
\]
In this case, Equation (\ref{p2:1}) holds true. Now, we assume
$n\geq3$. Without loss of generality, we further assume $\|b\|=1$.

Denote by
$\Omega_n=\{z=(z_1,\ldots,z_n)^T\in \Bbb R^n\mid \|z\|=1\}$
 the $n$-dimensional  unit sphere centered at $0$.
The spherical cap is defined as a set of points $y$ in $\Omega_n$ located within distance $\rho=\sqrt{2-2\cos\beta}$ from
a fixed point $x\in \Omega_n$:
\[
{\rm Cap}_n(x,\beta)=\{y\in \Omega_n\mid \|y-x\|\le \rho \}=
\{y\in \Omega_n\mid x^Ty\ge \cos\beta\}.
\]
Due to the rotational invariance,  the Lebesgue measure
of the  spherical cap ${\rm Cap}_n(x,\beta)$ is independent of the location of $x$. We denote it by $C_n(\beta)$. More precisely, we have
\[
C_n(\beta)=k_{n-1}{\int_0}^\beta\sin^{n-2}\theta {\rm d}\theta,
\]
where
\begin{eqnarray*}
k_{n-1}&=&\left\{\begin{array}{ll}\frac{(2\pi)^\frac{n-1}{2}}{(n-3)!!},
&n=3,5,\ldots,\\
2\cdot\frac{(2\pi)^\frac{n-2}{2}}{(n-3)!!},
&n=4,6,\ldots,\end{array}\right.
\end{eqnarray*}
see for example, Appendix B of \cite{TV}.
It is easy to see that $C_n(\pi)=2C_n(\pi/2)$ represents the Lebesgue measure
of the unit sphere $\Omega_n$.

Define $\widetilde{y}=(\eta_1/\sqrt{n},\ldots,\eta_n/\sqrt{n})^T$.  We have $\widetilde{y}\in \Omega_n$ and
\[
{\rm Pr}\left( b^T\eta \geq \alpha \right)
={\rm Pr}\left( b^T\widetilde{y} \geq \frac{\alpha}{\sqrt{n}} \right)=C_n\left(\arccos{\frac{\alpha}{\sqrt{n}}}\right)\bigg/\left(2C_n(\pi/2)\right),
\]
which completes the proof of (\ref{p2:1}) by replacing $\theta$ with $\arccos t$ in the two integrations.

The proof of the tail estimation (\ref{ie:1}) is tedious. We leave it in Appendix B.
\end{proof}

Based on the strict monotonicity of $S(n,\alpha)$ in
terms of $\alpha\in(0,\sqrt{n})$ and the inequality (\ref{ie:1}),
we immediately have the following result.
\begin{cor}\label{lem3}
For any fixed $n\ge 2$,
$S(n,\alpha)$ in terms of $\alpha\in(0,\sqrt{n})$ has an inverse function $S^{-1}(n,\beta)$ satisfying $S(n,S^{-1}(n,\beta))=\beta$ and
\begin{equation}
S^{-1}(n,\beta)< \sqrt{\frac{20}{9}\ln\left(\frac{1}{\beta}\right)},~
\forall\beta\in(0,0.5).\label{ie:2}
\end{equation}
\end{cor}

Following from Theorem \ref{lem1}, we can get the following tail estimation.
\begin{cor}\label{lem2}
Let $b^i\in \Bbb R^n,~i=1,\ldots,m$ and $I:=\{i\in\{1,\ldots,m\}\mid~\|b^i\|>0\}$.  Let $\eta$ be
uniformly
distributed over the sphere of radius $\sqrt{n}$ in $\Bbb R^n$. Then,
for any fixed $ \rho\in(0, 1)$,
\[
{\rm Pr}\left((b^i)^T\eta< S^{-1}\left(n,\frac{\rho}{m}\right)\|b^i\|,~i\in I\right)\ge 1-\rho>0.
\]
\end{cor}
\begin{proof}
Let $|I|$ be the number of elements in $I$. If $I=\emptyset$, there is nothing to prove. So, we assume $1\le |I|\le m$.
Let $\beta=S^{-1}\left(n,\frac{\rho}{m}\right)$. According to Theorem \ref{lem1}, we have
\[
{\rm Pr}\left((b^i)^T\eta\geq \beta \|b^i\|\right)= \frac{\rho}{m},~i\in I.
\]
Then, it holds that
\begin{eqnarray*}
{\rm Pr}\left((b^i)^T\eta< \beta \|b^i\|,~i\in I\right)&=&1-{\rm Pr}(\left(b^i)^T\eta\geq \beta \|b^i\|~{\rm for~some}~i\in I\right)\\
&\geq& 1-\sum_{i\in I}{\rm Pr}\left((b^i)^T\eta\geq \beta \|b^i\|\right)\\
&=&1-\sum_{i=1}^{|I|}\frac{\rho}{m}=1-\frac{|I|}{m}\rho
\ge1-\rho>0.
\end{eqnarray*}
The proof is complete.
\end{proof}

\subsection{A new approximation algorithm for (${\rm P_{ball}}$)}
In this subsection, we propose a new approximation algorithm for solving $({\rm P_{ball}})$ as follows and then show that the approximation bound is  $\frac{1-O(\sqrt{\ln(m)/n})}{2}$, which positively answers
the first open question raised in \cite{HA13}.

~

\begin{center}
\fbox{\shortstack[l]{
{\bf A new approximation algorithm for $({\rm P_{ball}})$}\\
1.~ Input $\rho\in(0,1)$ and $x^i$ for  $i=1,\ldots,m$. Let $\alpha=S^{-1}\left(n,\frac{\rho}{m}\right)$. \\
2.~  Repeatedly uniformly generate  $\zeta\in\Bbb R^n$ on the surface  of an $n$-dimensional \\~~~~~unit sphere
 until  $\sqrt{n}(x^i)^T\zeta< \alpha\|x^i\|$  for $i\in\{1,\ldots,m\}\setminus\{k\mid \|x^k\|=0\}$.\\
3.~ Output $\widetilde{x}=\zeta.$
}}
\end{center}

~

\begin{thm}\label{thm05}
For the solution $\widetilde{x}$ returned by the above algorithm,  we have
\[
v({\rm CR_{ball}})\geq v({\rm P_{ball}})\geq f(\widetilde{x})>r\cdot v({\rm CR_{ball}})\geq r\cdot v({\rm P_{ball}}),
\]
where the approximation bound $r$ satisfies that
\[
r=\frac{1-\frac{S^{-1}\left(n,\frac{\rho}{m}\right)}{\sqrt{n}}}{2}>
\frac{1-\sqrt{\frac{20}{9n}\ln\left(\frac{m}{\rho}\right)}}{2}.
\]
\end{thm}
\begin{proof}
According to the definition of $\widetilde{x}$ and $\alpha$,
we can verify that for $i=1,\ldots,m$ and $\|x^i\|>0$, it holds that
\begin{eqnarray}
\|\widetilde{x}-x^i\|^2&=&\|\widetilde{x}\|^2-2(x^i)^T\widetilde{x}+\|x^i\|^2\nonumber\\
&=&\|\widetilde{x}\|^2- 2(x^i)^T\zeta+\|x^i\|^2\nonumber\\
&>&\|\widetilde{x}\|^2-\frac{2\alpha}{\sqrt{n}}\|x^i\|+\|x^i\|^2\nonumber\\
&=&1-\frac{2\alpha}{\sqrt{n}}\|x^i\|+\|x^i\|^2\label{num:1}\\
&\ge&\left(1-\frac{\alpha}{\sqrt{n}}\right)\left(1+\|x^i\|^2\right),\label{num:2}
\end{eqnarray}
where the last inequality (\ref{num:2}) follows from the fact
\begin{equation}
2\|x^i\|\le 1+\|x^i\|^2.\label{num:3}
\end{equation}
If there is an index $k$ such that $\|x^k\|=0$, then
\begin{equation}
\|\widetilde{x}-x^k\|^2=1>
\left(1-\frac{\alpha}{\sqrt{n}}\right)\left(1+\|x^k\|^2\right).\label{0:1}
\end{equation}
Let $(x^*,\zeta^*)$ be an optimal solution of ${\rm (CR_{ball})}$. Then, for $i=1,\ldots,m$ and $\|x^i\|>0$, we have
\begin{eqnarray}
{\zeta^*}&\leq&{\omega_i}(1-2(x^i)^Tx^*+\|x^i\|^2)\nonumber\\
&\le&{\omega_i}(1+2\|x^i\|\cdot\|x^*\|+\|x^i\|^2)\label{num:4}\\
&\le&{\omega_i}(1+2\|x^i\|+\|x^i\|^2)\label{num:6}\\
&\leq&2{\omega_i}(1+\|x^i\|^2)\label{num:5},
\end{eqnarray}
where the second inequality (\ref{num:4}) follows from the Cauchy-Schwartz inequality and the last inequality (\ref{num:5}) holds due to (\ref{num:3}).
In the case that there is an index $k$ such that $\|x^k\|=0$, we have
\begin{equation}
{\zeta^*}\leq {\omega_k}(1-2(x^k)^Tx^*+\|x^k\|^2)=\omega_k
<2{\omega_k}(1+\|x^k\|^2).\label{0:2}
\end{equation}
Combining the inequalities (\ref{num:2}), (\ref{0:1}), (\ref{num:5}) and (\ref{0:2}) yields
\[
\min_{i=1,\ldots,m}\omega_i\|\widetilde{x}-x^i\|^2> \frac{1-\frac{\alpha}{\sqrt{n}}}{2}\zeta^*.
\]
Thus, we get
\[
v({\rm CR_{ball}})\geq v({\rm P_{ball}})\geq f(\widetilde{x})> \frac{1-\frac{\alpha}{\sqrt{n}}}{2}v({\rm CR_{ball}})\geq \frac{1-\frac{\alpha}{\sqrt{n}}}{2}v({\rm P_{ball}}).
\]
According to the definition of $\alpha$, it follows from (\ref{ie:2}) in
 Corollary  \ref{lem3} that
\begin{equation}
\alpha=S^{-1}\left(n,\frac{\rho}{m}\right)<\sqrt{\frac{20}{9}\ln\left(
\frac{m}{\rho}\right)}.
\label{alpha}
\end{equation}
The proof is complete.
\end{proof}

Theorem \ref{thm05} guarantees that (${\rm P_{ball}}$) admits a $1/2$ asymptotic approximation bound as $\frac{n}{\ln(m)}$ increases to infinity.
\begin{cor}
Let $\widetilde{x}$ be the solution returned by the above algorithm. Then, it holds that
\[
\liminf_{\frac{n}{\ln(m)}\rightarrow+\infty}\frac{v({\rm P_{ball}})}{v({\rm CR_{ball}})}
\ge \liminf_{\frac{n}{\ln(m)}\rightarrow+\infty}\frac{f(\widetilde{x})}{v({\rm CR_{ball}})}
\ge \frac{1}{2}.
\]
\end{cor}

Finally, we show that Theorem \ref{thm05} can be further refined with some additional information on $\|x^i\|$ ($i=1,\ldots,m$).

\begin{thm}
Let $d=\min\limits_{i=1,\ldots,m} \|x^i\|$. For the solution $\widetilde{x}$ returned by the above algorithm, we have
\[
v({\rm P_{ball}})\geq
f(\widetilde{x})> \left(\frac{\nu}{2+\nu}- \frac{2}{2+\nu}\sqrt{\frac{20}{9n}\ln\left(\frac{m}{\rho}\right)}
 \right)v({\rm CR_{ball}}),
\]
where
\[
\nu=\min_{t\ge d}\left\{t+1/t\right\}=\left\{\begin{array}{ll}
d+1/d,&{\rm if}~d> 1,\\
2,&~{\rm otherwise}.\end{array}\right.
\]
\end{thm}
\begin{proof}
If $d=0$, that is, there is an index $k$ such that $\|x^k\|=0$, then it holds that
\begin{equation}
\omega_k\|\widetilde{x}-x^k\|^2=\omega_k= {\omega_k}(1-2(x^k)^Tx^*+\|x^k\|^2)\ge
v({\rm CR_{ball}}),\label{0:3}
\end{equation}
which completes the proof when $m=1$. Generally, it follows from (\ref{num:1}), (\ref{num:6}) and (\ref{0:3}) that
\begin{eqnarray}
f(\widetilde{x})=\min_{i=1,\ldots,m}\omega_i\|\widetilde{x}-x^i\|^2
&>&\min_{i=1,\ldots,m}\left\{\frac{1-\frac{2\alpha}{\sqrt{n}}\|x^i\|+
\|x^i\|^2}{1+2\|x^i\|+\|x^i\|^2}\right\}\cdot v({\rm CR_{ball}})\nonumber\\
&\ge&\min_{t\ge d }\left\{\frac{1-\frac{2\alpha}{\sqrt{n}}t+
t^2}{1+2t+t^2}\right\}\cdot v({\rm CR_{ball}})\nonumber\\
&=&\min_{t\ge d }\left\{1- \left(2+\frac{2\alpha}{\sqrt{n}}\right)\frac{t}{(1+t)^2}\right\}\cdot v({\rm CR_{ball}})\nonumber\\
&=& \left(1- \left(2+\frac{2\alpha}{\sqrt{n}}\right)\frac{1}{2+\min_{t\ge d }\left\{t+1/t\right)}\right) \cdot v({\rm CR_{ball}}).\label{mind}
\end{eqnarray}
Plugging (\ref{alpha}) into (\ref{mind}) completes the proof.
\end{proof}

\section{Numerical experiments}
In this section, we present the numerical comparison between our new
approximation algorithm and the general algorithm proposed in \cite{HA13} (see also Section \ref{sec} in this paper) for solving $({\rm P_{ball}})$.

\subsection{Algorithm implementation}
We first give some details for implementing the two algorithms for solving $({\rm P_{ball}})$.


In the second step of the general algorithm in \cite{HA13}, according to  Theorem \ref{thm01}, we do not need to solve $({\rm SDP_{ball}})$. Instead, we first solve the second-order cone programming relaxation $({\rm CR_{ball}})$ using SDPT3 within CVX \cite{GrB} and then construct the optimization solution of $({\rm SDP_{ball}})$ based on (\ref{yy}).

While in our new algorithm, none of the convex relaxations needs to be solved.
In our first step, the value $\alpha=S^{-1}\left(n,\frac{\rho}{m}\right)$ is obtained by solving the nonlinear equation $S(n,\alpha)=\frac{\rho}{m}$ with respect to $\alpha$.
In the second step,  there are many approaches to randomly and uniformly generate points on the surface of an $n$-dimensional unit  hypersphere, for example, see \cite{GZW} and references therein. As pointed out by one referee, there is a much simpler method \cite{HJLZ14} as follows.
The existence of such a point is guaranteed by Corollary \ref{lem2}.

~

\begin{center}
\fbox{\shortstack[l]{
{\bf Uniform random sampling of points on a unit-hypersphere}\\
1.~Generate $\mu\in\Bbb R^n$ with independent standard  normal random
\\~~~~components $\mu_i$,~$i=1,\ldots,n$.\\
2.~Output $\mu/\|\mu\|$.}}
\end{center}

~

\subsection{Numerical results}

We do numerical experiments on $25$ random instances of dimension $n=5$, where all weights $\omega_1,\ldots,\omega_n$ are equal to $1$ and the number of input points $m$ varies from $6$ to $30$, as the instances when $m\le n$ are trivial to solve  according to Theorem \ref{poly}. All the numerical tests are implemented in MATLAB R2013b and run on a laptop with $2.3$ GHz processor and $4$ GB  RAM.

All the input points $x^i~(i=1,\ldots,m)$ with $m=6,7,\ldots,30$ orderly form an $n\times 450$ matrix. We randomly generate this matrix using the following Matlab scripts:
\begin{verbatim}
              rand('state',0); X = 2*rand(n,450)-1;
\end{verbatim}
and then independently run each of the two random algorithms $10$ times for each instance using the same setting $\rho=0.9999$. As references, we globally solve the 25 instances of $({\rm P_{ball}})$ by calling the  branch-and-bound type algorithm proposed in \cite{Lu} for the general quadratic constrained quadratic programming problem.

We report the numerical results in  Table \ref{tab1}.
The columns `$v{\rm(P)}$' and `$v{\rm(CR)}$' present the optimal objective function values of the 25 instances of ${\rm(P_{ball})}$ and the convex relaxation $({\rm CR_{ball}})$, respectively. The next two columns present the statistical results over the $10$ runs of the general algorithm proposed in \cite{HA13} and those of our new algorithm, respectively. The subcolumns `$v_{\max}$', `$v_{\min}$' and `$v_{\rm ave}$' give the best, the worst and the average objective function values found among $10$ tests, respectively. The  subcolumns `l.b.' present the  lower bounds on the optimal objective function values estimated by the two algorithms in comparison, i.e.,
$\frac{1-\sqrt{2\ln(m/\rho)\gamma_1^*}}{2}v({\rm SDP_{ball}})$ in \cite{HA13} (see also Theorem \ref{thm03} in this paper) and $\frac{1-S^{-1}\left(n,\frac{\rho}{m}\right)/\sqrt{n}}{2}v({\rm CR_{ball}})$ established in Theorem \ref{thm05}.
We highlight each statistical result for the general algorithm in bold, if it is better (i.e., larger) than the corresponding result for our new algorithm.

Table \ref{tab1} shows that the qualities of the solutions returned by our algorithm are generally higher than those obtained by the general algorithm in \cite{HA13}. In particular, there are a little more than half of instances such that the objective function values of the worst solutions returned by the general algorithm are even less than our new estimated lower bounds on the optimal values. Thus, our approximation solutions are more robust.
Moreover, the lower bounds on the positive optimal values estimated by the general algorithm in \cite{HA13} are all negative and have been significantly improved by our new estimated lower bounds.

 \begin{table}\tiny\caption{ Numerical results for $n=5$.}
\begin{center}
 \begin{tabular}[h]{|l|cc|cccc|cccc|}
 \hline
  \multirow{2}{*}{m} & \multirow{2}{*}{$v{\rm(P)}$} & \multirow{2}{*}{$v{\rm(CR)}$} & \multicolumn{4}{|c|}{the general algorithm in \cite{HA13}} &
  \multicolumn{4}{|c|}{our new algorithm}\\
\cline{4-11}
 & & &$v_{\max}$& $v_{\min}$&$v_{\rm ave}$
&l.b.&$v_{\max}$ & $v_{\min}$&$v_{\rm ave}$
&l.b.\\ \hline
6  &	2.74 &	2.74      &	{\bf2.74} &	1.12 &	{\bf1.80} &	-0.43     &	2.01 &	1.42 &	1.78 &	0.71 \\
7  &	2.50 &	2.50      &	{\bf1.94} &	0.25 &	1.03 &	-0.67     &	1.88 &	0.98 &	1.54 &	0.59 \\
8  &	1.80 &	1.80      &	{\bf1.78} &	0.45 &	1.00 &	-0.74     &	1.44 &	0.77 &	1.07 &	0.40 \\
9  &	2.45 &	2.45      &	{\bf2.23} &	0.25 &	0.71 &	-0.47     &	1.59 &	0.75 &	1.15 &	0.51 \\
10 &	2.29 &	2.31      &	1.45      &	0.29 &	0.71 &	-0.62     &	1.69 &	0.79 &	1.12 &	0.45 \\
11 &	2.22 &	2.22      &	1.80      &	0.33 &	1.00 &	-0.80     &	2.08 &	0.73 &	1.18 &	0.41 \\
12 &	2.21 &	2.21      &	{\bf2.21} &	0.37 &	{\bf1.14} &	-0.58     &	1.21 &	0.60 &	0.90 &	0.39 \\
13 &	1.73 &	1.74      &	{\bf1.39} &	0.61 &	0.91 &	-0.70     &	1.30 &	0.66 &	1.01 &	0.30 \\
14 &	1.81 &	1.81      &	{\bf1.81} &	0.24 &	0.89 &	-0.36     &	1.51 &	0.57 &	0.89 &	0.30 \\
15 &	2.02 &	2.19      &	1.44      &	0.09 &	0.80 &	-0.20     &	1.48 &	0.68 &	0.96 &	0.35 \\
16 &	1.89 &	1.89      &	{\bf1.45} &	0.31 &	0.71 &	-0.84     &	1.42 &	0.62 &	0.97 &	0.29 \\
17 &	2.13 &	2.13      &	{\bf2.11} &	0.44 &	0.76 &	-0.60     &	1.57 &	0.84 &	1.18 &	0.31 \\
18 &	1.65 &	1.93      &	0.88      &	0.23 &	0.50 &	-0.37     &	0.91 &	0.51 &	0.74 &	0.28 \\
19 &	1.72 &	1.93      &	1.02      &	0.27 &	0.61 &	-0.32     &	1.28 &	0.63 &	0.94 &	0.27 \\
20 &	2.09 &	2.51      &	{\bf1.69} &	0.43 &	0.85 &	-0.13     &	1.42 &	0.75 &	1.06 &	0.34 \\
21 &	1.92 &	2.07      &	0.94      &	0.27 &	0.53 &	-0.76     &	1.22 &	0.55 &	0.82 &	0.27 \\
22 &	1.77 &	2.20      &	0.80      &	0.44 &	0.62 &	-0.37     &	1.41 &	0.64 &	0.98 &	0.28 \\
23 &	1.83 &	2.13      &	{\bf1.18} &	0.39 &	0.67 &	-0.26     &	1.04 &	0.52 &	0.75 &	0.27 \\
24 &	1.76 &	1.85      &	0.76      &	0.49 &	0.59 &	-0.31     &	1.28 &	0.59 &	0.78 &	0.23 \\
25 &	1.73 &	1.92      &	1.05      &	0.20 &	0.45 &	-0.35     &	1.26 &	0.43 &	0.87 &	0.23 \\
26 &	1.53 &	1.82      &	{\bf1.17} &	0.05 &	0.45 &	-0.43     &	1.08 &	0.53 &	0.74 &	0.21 \\
27 &	1.56 &	1.88      &	0.85      &	0.25 &	0.49 &	-0.21     &	1.23 &	0.43 &	0.73 &	0.22 \\
28 &	1.73 &	1.85      &	{\bf1.46} &	0.18 &	0.54 &	-0.41     &	1.42 &	0.45 &	0.78 &	0.21 \\
29 &	2.38 &	2.39      &	{\bf1.83} &	0.24 &	0.71 &	-0.87     &	1.53 &	0.49 &	1.02 &	0.27 \\
30 &	1.60 &	1.82      &	1.07      &	0.14 &	0.50 &	-0.57     &	1.24 &	0.47 &	0.84 &	0.20 \\
 \hline
 \end{tabular}
 \end{center}\label{tab1}
 \end{table}

\section{Conclusions}
For the $n$-dimensional ball-constrained weighted maximin dispersion problem $({\rm P_{ball}})$, complexity and approximation bounds  remain unknown.
In this paper, we propose a new second-order cone programming (SOCP) relaxation for $({\rm P_{ball}})$, which is shown to be equivalent to the standard semidefinite programming (SDP) relaxation. Besides, applying the new relaxation approach to the box-constrained weighted maximin dispersion problem $({\rm P_{box}})$ yields a linear programming relaxation, which is also equivalent to the corresponding SDP relaxation. Furthermore, with the help of the new SOCP relaxation, we show that $({\rm P_{ball}})$ is polynomially solvable if $m\le n$, where $m$ is the number of given points. A more general sufficient condition is also provided. In general, we have proved that $({\rm P_{ball}})$ is NP-hard.
Then, we propose a new randomized approximation algorithm for solving $({\rm P_{ball}})$, which provides an approximation bound of $\frac{1-O(\sqrt{\ln(m)/n})}{2}$.
Notice that $({\rm P_{ball}})$ and $({\rm P_{box}})$ are the
two  special cases of $({\rm P}_{\chi_p})$, where $\chi_p=\{x\mid~\|x\|_p\le 1\}$
and $\|x\|_p:=\left(\sum_{i=1}^n |x_i|^p\right)^{\frac{1}{p}}$ is the $\ell_p$-norm of $x$. It is the future work to study the complexity and approximation bound
of $({\rm P}_{\chi_p})$ for  $p\in[1,+\infty)$.

\section*{Acknowledgments}
The authors thank Dr. Cheng Lu for sharing the MATLAB code for globally solving the general quadratic constrained quadratic programming problem, which is used to get the optimal values of our test examples. The authors are grateful to the two anonymous referees for their valuable comments.


\begin{thebibliography}{99}

\bibitem{Ben2002}
{\sc A. Ben Tal, A. Nemirovski, and C.~Roos}, {\em Robust solutions of
  uncertain quadratic and conic-quadratic problems}, SIAM J. Optim., 13 (2002),
  pp. 535--560.

  \bibitem{Conn}
{\sc  A. R. Conn, N. I. M. Gould and Ph. L. Toint},  {\em Trust-Region
 Methods},
Number 01 MPS-SIAM Series on Optimization, Philadelphia, SIAM,  2000.

\bibitem{DW}
{\sc B. Dasarathy and L. J. White}, {\em a  maximin location problem,} Oper. Res., 28 (1980), pp.1385--1401.




 \bibitem{TV}
 {\sc T. Ericson and V. Zinoviev}, {\em Codes on Euclidean Spheres},
 North-Holland Mathematical Library, 2001.

\bibitem{GJ}
{\sc M. R. Garey and D. S. Johnson}, {\em Computers and Intractability: A guide to the theory of
NP-completeness}, Freeman, San Francisco, 1979.

\bibitem{GrB}
{\sc M. Grant and S. Boyd}, {\em CVX: Matlab software for disciplined convex programming, version 2.1}, http://cvxr.com/cvx, June 2015.

\bibitem{GZW}
{\sc G. Guralnik, C. Zemach and T. Warnock}, {\em An algorithm for uniform random sampling of
points in and on a hypersphere}, Inform. Process. Lett., 21 (1985), pp. 17--21.

\bibitem{Hag}
{\sc W. W. Hager}, {\em Updating the inverse of a matrix}, SIAM Rev., 31(2)  (1989) pp. 221--239.

\bibitem{HA13}
{\sc S. Haines, J. Loeppky, P. Tseng and X. Wang},
{\em Convex relaxations of the weighted maxmin dispersion problem},
SIAM J. Optim.,
23(4) (2013), pp. 2264--2294.



\bibitem{HJLZ14}
{\sc S. He, B. Jiang, Z. Li and S. Zhang}, {\em Probability bounds for polynominal functions in random variables,} Math. Oper. Res., 39 (2014), 889--907.
\bibitem{KN08}
{\sc S. Khot and A. Naor}, {\em Linear equations modulo 2 and the $L_1$ diameter of convex bodies,}
SIAM J. Comput., 38 (2008), 1448--1463.

\bibitem{JMY}
{\sc M. E. Johnson, L. M. Moore and D. Ylvisaker}, {\em Minimax mand maximin distance designs,}
J. Satist. Plan. Inference, 26 (1990), pp. 131--148.

\bibitem{Lu}
{\sc C. Lu, Z. B. Deng and J. Zhou}, {\em Sensitive-eigenvector-based approach to solving
quadratically constrained quadratic programs}, submitted for publication, 2015.

\bibitem{Sc}
{\sc R. Schaback}, {\em Multivariate Interpolation and Approximation by Translates of a Basis Function}, in Approximation Theory VIII,  C. K. Chui and L. L. Schumaker eds., World Scientific, Singapore, 1995, pp. 491--514.

\bibitem{Te}
{\sc Y. Tenne, K. Izui and S. Nishiwaki}, {\em A hybrid model-classifier framework for managing prediction uncertainty in expensive optimisation problems}, Internat. J. Systems Sci., 43(7) (2012), pp. 1305--1321.

\bibitem{WX}
{\sc S. Wang and Y. Xia}, {\em Uniform quadratic optimization and extensions}, presented at 8th International Congress on Industrial and Applied Mathematics, Bejing, 2015,
http://arxiv.org/abs/1508.01000

\bibitem{W}
{\sc D. J. White}, {\em A heuristic approach to a weighted maxmin dispersion problem,} IMA J. Math. Appl. Business
and Industry, 7 (1996), pp. 219--231.
\end{thebibliography}

\section{Appendix}

\subsection{Appendix A: supplementary proof of Theorem \ref{thm32}}
Since the root $t^*$ can be polynomially approximated, let $\widetilde{L}$ be an invertible rational matrix and $\widetilde{\Lambda}$ be a rational diagonal matrix such that they are sufficiently close to $\Lambda$ (\ref{Lmbd}) and  $L$ (\ref{LL}), respectively. Suppose  it holds that
\begin{eqnarray}
&&\left|\|\widetilde{L}_i\|^2-\|L_i\|^2\right|=
\left|\|\widetilde{L}_i\|^2-1\right|\le \epsilon_1, \label{np:e1}\\
&&\left|{\rm Tr}(\widetilde{\Lambda})-{\rm Tr}(\Lambda)\right|\le \epsilon_2, \label{np:e2}\\
&& \| \widetilde{\Lambda}-aa^T-\widetilde{L}^{-T}\widetilde{L}^{-1}\|\le
\frac{\epsilon_2}{8n},\label{np:e3}\\
&&  \widetilde{\Lambda}-aa^T \succ 0,\nonumber
\end{eqnarray}
where $\epsilon_2\in(0,\frac{1}{2})$,
\begin{equation}
\epsilon_1\in\left(0,\frac{\epsilon_2}{4\sqrt{U+\epsilon_2/8}
\sqrt{U+3\epsilon_2/8}\left(\sqrt{U+\epsilon_2/8}+\sqrt{U+3\epsilon_2/8}\right)}
\right),\label{eps1}
\end{equation}
and $U$ is any constant satisfying $U\ge{\rm Tr}(\Lambda)$.

Let $v_a:=\min_{x\in\{-1,1\}^n}(a^Tx)^2$. (PP) (\ref{pp}) has no solution if and only $v_a\ge1$.
Since the elements of $a$ are all integers, $v_a$ remains an integer. Then, it holds that
\begin{eqnarray*}
&&v_a\ge 1~\Longleftrightarrow~ v_a>\frac{\epsilon_2}{2}, \\
&&v_a< 1~\Longleftrightarrow~ v_a< \epsilon_2.
\end{eqnarray*}
Let the Hessian matrix of (BQP) (\ref{bqp:1}) be $Q=\frac{1}{4}\left(\widetilde{\Lambda}-aa^T\right)$.
Let (${\rm \overline{BQP}}$) and (${\rm \overline{QCQP}}$) be the problems obtained by
replacing the Hessian matrix $Q$ with $\frac{1}{4}\widetilde{L}^{-T}\widetilde{L}^{-1}$ in (BQP) (\ref{bqp:1}) and (QCQP) (\ref{qcqp:1}), respectively.

Clearly,  it holds  that
\[
v{\rm(BQP)}= \frac{1}{4}{\rm Tr}(\widetilde{\Lambda})-\frac{1}{4}v_a.
\]
Therefore, (PP) (\ref{pp}) has a solution if and only if $v{\rm(BQP)}> \frac{1}{4}{\rm Tr}(\widetilde{\Lambda})-\frac{\epsilon_2}{4}$  and (PP) (\ref{pp}) has no solution if and only if $v{\rm(BQP)}< \frac{1}{4}{\rm Tr}(\widetilde{\Lambda})-\frac{\epsilon_2}{8}$.

According  to Lemma \ref{lem4}, we have
\begin{equation}
v({\rm \overline{QCQP}})=\left\{\begin{array}{cl}
2-\frac{1}{\sqrt{v({\rm \overline{BQP}})}},& v({\rm \overline{BQP}})\ge 1,\\
1,& v({\rm \overline{BQP}})< 1.
\end{array}
\right. \label{np:5}
\end{equation}
According to (\ref{np:e3}), we have
\[
\left|x^TQx-x^T\left(\frac{1}{4}\widetilde{L}^{-T}\widetilde{L}^{-1}\right)x\right|
\le \frac{1}{4}\| \widetilde{\Lambda}-aa^T-\widetilde{L}^{-T}\widetilde{L}^{-1}\|\cdot\|x\|^2 \le \frac{\epsilon_2}{32},~\forall x\in[-1,1]^n,
\]
which implies that
\begin{equation}
\left|v({\rm BQP})-v({\rm \overline{BQP}})\right|\le \frac{\epsilon_2}{32}.
\label{qp32}
\end{equation}

Without loss of generality, we assume $n\ge 5$. Then, we have
\begin{equation}
{\rm Tr}(\widetilde{\Lambda}) \ge {\rm Tr}(\Lambda)-\epsilon_2> {\rm Tr}(\Lambda)-\frac{1}{2}>n-\frac{1}{2}\ge \frac{9}{2}.
\label{np:0}
\end{equation}
According to (\ref{np:5}), the inequality
\begin{equation}
v({\rm \overline{BQP}})>\frac{1}{4}{\rm Tr}(\widetilde{\Lambda})-\frac{7\epsilon_2}{32}(>1)
\label{qp7}
\end{equation}
holds if and only if
\begin{equation}
v({\rm \overline{QCQP}})>2-\frac{2}{\sqrt{{\rm Tr}(\widetilde{\Lambda})-7\epsilon_2/8}}.
\label{np:1}
\end{equation}
Similarly, the following inequality
 \begin{equation}
 v({\rm \overline{BQP}})<\frac{1}{4}{\rm Tr}(\widetilde{\Lambda})-\frac{5\epsilon_2}{32}
 \label{qp5}
\end{equation}
holds if and only if either $v({\rm \overline{QCQP}})<2-\frac{2}{\sqrt{{\rm Tr}(\widetilde{\Lambda})- 5\epsilon_2/8}}$ or $v({\rm \overline{QCQP}})= 1$,
which is further equivalent to
\begin{equation}
v({\rm \overline{QCQP}})<2-\frac{2}{\sqrt{{\rm Tr}(\widetilde{\Lambda})- 5\epsilon_2/8 }},
\label{np:2}
\end{equation}
since we always have $1<2-\frac{2}{\sqrt{{\rm Tr}(\widetilde{\Lambda})-5\epsilon_2/8}}$ according  to (\ref{np:0}) and the inequality $\epsilon_2<\frac{1}{2}<\frac{4}{5}$.

By introducing $y=\frac{1}{2}\widetilde{L}^{-1}x$, we can reformulate
(${\rm \overline{QCQP}}$) as
\begin{eqnarray*}
 &\max_{y,~s}& ~y^Ty+s \\
&{\rm s.~t.}& 2\widetilde{L}_iy\le 1-s,~i=1,\ldots,n,\\
& & -2\widetilde{L}_iy\le 1-s,~i=1,\ldots,n,\\
&&y^Ty\le 1.\nonumber
\end{eqnarray*}
On the other hand, setting $\omega_1=\ldots=\omega_n=1$, $m=2n$ and
\begin{eqnarray*}
&&x^i=\widetilde{L}_i^T,~i=1,\ldots,n,\\
&&x^i=-\widetilde{L}_{i-n}^T,~i=n+1,\ldots,2n,
\end{eqnarray*}
in ${\rm (P_{ball})}$ yields
\begin{eqnarray*}
{\rm (P_{ball}^L)}~~&\max_{x,~s}& ~x^Tx+s\\
&{\rm s.~t.}&
2\widetilde{L}_ix\le \|\widetilde{L}_i\|^2-s,~i=1,\ldots,n,\\
& & -2\widetilde{L}_ix\le \|\widetilde{L}_i\|^2-s,~i=1,\ldots,n,\\
&&x^Tx\le 1.
\end{eqnarray*}
Therefore, we have
\begin{eqnarray*}
v{\rm (P_{ball}^L)}&=&\max_{x^Tx\le 1}\{x^Tx+\min_{i=1,\ldots,n}\{\|\widetilde{L}_i\|^2-2\widetilde{L}_ix,
\|\widetilde{L}_i\|^2+2\widetilde{L}_ix\}\}\\
&\le&\max_{x^Tx\le 1}\{x^Tx+\min_{i=1,\ldots,n}\{1+\epsilon_1-2\widetilde{L}_ix,
1+\epsilon_1+2\widetilde{L}_ix\}\}\\
&=&\epsilon_1+\max_{x^Tx\le 1}\{x^Tx+\min_{i=1,\ldots,n}\{1-2\widetilde{L}_ix,
1+2\widetilde{L}_ix\}\}\\
&=&\epsilon_1+v({\rm \overline{QCQP}}),
\end{eqnarray*}
where the inequality follows from (\ref{np:e1}). Similarly, we can show that
\[
v{\rm (P_{ball}^L)}\ge -\epsilon_1+v({\rm \overline{QCQP}}).
\]
Now, if the inequality
\begin{equation}
v{\rm (P_{ball}^L)}>2-\frac{2}{\sqrt{{\rm Tr}(\widetilde{\Lambda})-7\epsilon_2/8}}+\epsilon_1, \label{np:lst1}
\end{equation}
holds, then we obtain (\ref{np:1}). It follows from (\ref{qp32}) and (\ref{qp7}) that
\[
v{\rm (BQP)}>v({\rm \overline{BQP}})-\frac{\epsilon_2}{32}>\frac{1}{4}{\rm Tr}(\widetilde{\Lambda})-\frac{7\epsilon_2}{32}-\frac{\epsilon_2}{32}=\frac{1}{4}{\rm Tr}(\widetilde{\Lambda})-\frac{\epsilon_2}{4}.
\]
Consequently, (PP) (\ref{pp}) has a solution. Similarly, suppose
\begin{equation}
v{\rm (P_{ball}^L)}<2-\frac{2}{\sqrt{{\rm Tr}(\widetilde{\Lambda})-5\epsilon_2/8}}-\epsilon_1,\label{np:lst2}
\end{equation}
then (\ref{np:2}) holds. According to (\ref{qp32}) and (\ref{qp5}), we obtain
\[
v{\rm (BQP)}<v({\rm \overline{BQP}})+\frac{\epsilon_2}{32}<\frac{1}{4}{\rm Tr}(\widetilde{\Lambda})-\frac{5\epsilon_2}{32}+\frac{\epsilon_2}{32}=
\frac{1}{4}{\rm Tr}(\widetilde{\Lambda})-\frac{\epsilon_2}{8}.
\]
It implies that (PP) (\ref{pp}) has no solution.

Since the upper bound in the definition $\epsilon_1$ (\ref{eps1}) is a decreasing function with respect to $U$, and according to the inequalities
\[
U\ge{\rm Tr}(\Lambda) \ge {\rm Tr}(\widetilde{\Lambda})-\epsilon_2,
\]
where the second inequality follows from (\ref{np:e2}), we have
\[
\epsilon_1<  \frac{1}{\sqrt{{\rm Tr}(\widetilde{\Lambda})-7\epsilon_2/8}}-
\frac{1}{\sqrt{{\rm Tr}(\widetilde{\Lambda})-5\epsilon_2/8}}.
\]
It implies that one of the two inequalities (\ref{np:lst1}) and (\ref{np:lst2}) must be satisfied.
Therefore, we  conclude that ${\rm (P_{ball})}$ is NP-hard.

\subsection{Appendix B: proof of the tail estimation (\ref{ie:1}) in Theorem \ref{lem1}}
For any fixed $n$, since
\begin{equation}
\frac{\partial S(n,\alpha)}{\partial \alpha}
=-\frac{(1-\alpha^2/n)^{\frac{n-3}{2}}}
{2\sqrt{n}\int_{0}^{1}\left(\sqrt{1-t^2}\right)^{n-3}{\rm d}t}
 <0,\label{s:de}
\end{equation}
$S(n,\alpha)$ is  strictly  decreasing in terms of $\alpha\in(0,\sqrt{n})$.

The proof of (\ref{ie:1}) is divided into
two parts: (a) $n\ge40$ and (b) $2\le n\le39$.

(a) Suppose $n\geq 40$. First, we have
\begin{eqnarray}
&&\int_0^{\frac{\alpha}{\sqrt{n}}}\left(\sqrt{1-t^2}\right)^{n-3}{\rm d}t\nonumber\\
&>&\int_0^{\frac{\alpha}{\sqrt{n}}}\left(1-\frac{\sqrt{n}-\sqrt{n-\alpha^2}}{\alpha}t\right)^{n-3}{\rm d}t\label{e1}\\
&=& \frac{n+\sqrt{n^2-\alpha^2n}}{\alpha(n-2)\sqrt{n}}\cdot\left(
1-\bigg(1-\frac{\alpha^2}{ n+\sqrt{n^2-\alpha^2n}} \bigg)^{n-2}\right)\nonumber\\
&>&\frac{1}{\alpha\sqrt{n}}\left(
1- \bigg(1-\frac{\alpha^2}{2n} \bigg)^{n\cdot \frac{n-2}{n}}\right) \nonumber\\
&\ge&\frac{1}{\alpha\sqrt{n}}\left(
1- \bigg(1-\frac{\alpha^2}{2n} \bigg)^{0.95n }\right) ~~\left({\rm as}~\frac{n-2}{n}\geq \frac{40-2}{40}=0.95\right) \nonumber\\
&>&\frac{1}{\alpha\sqrt{n}}\left(1-e^{-0.475\alpha^2}\right),\label{rel:4}
\end{eqnarray}
where the first inequality (\ref{e1}) holds as the strict concavity of $\sqrt{1-t^2}$ implies that
 \[
\sqrt{1-t^2}> 1-\frac{\sqrt{n}-\sqrt{n-\alpha^2}}{\alpha}t,~\forall~ t\in \left(0,\frac{\alpha}{\sqrt{n}}\right),
\]
and the last inequality (\ref{rel:4})  holds since $\left(1-\frac{1}{2n} \right)^{n}=e^{n\log\left(1-\frac{1}{2n} \right)}$ is an increasing function of $n$ and has a limit $e^{-\frac{1}{2}}$ when $n\rightarrow \infty$.

Secondly, we have
\begin{eqnarray}
&&\int_{\frac{\alpha}{\sqrt{n}}}^{1}\left(\sqrt{1-t^2}\right)^{n-3}{\rm d}t
\nonumber\\
&=&
\sum_{k=0}^{[{(1-\frac{\alpha}{\sqrt{n}})n}]}
\int_{\frac{\alpha}{\sqrt{n}}+\frac{k}{n}}^{\min(1,\frac{\alpha}{\sqrt{n}}
+\frac{k+1}{n})}\left(\sqrt{1-t^2}\right)^{n-3}{\rm d}t
\nonumber\\
&\leq & \sum_{k=0}^{[{(1-\frac{\alpha}{\sqrt{n}})n}]}
\frac{1}{n}\cdot\bigg(\sqrt{1-\bigg(\frac{\alpha}{\sqrt{n}}
+\frac{k}{n}\bigg)^2}\bigg)^{n-3}\label{e2}\\
&=& \frac{1}{n}\sum_{k=0}^{[{(1-\frac{\alpha}{\sqrt{n}})n}]}
\bigg(1-\bigg(\frac{\alpha}{\sqrt{n}}+\frac{k}{n}\bigg)^2\bigg)^{\frac{n-3}{2}}\nonumber\\
&<&{\frac{1}{n}}\sum_{k=0}^{[{(1-\frac{\alpha}{\sqrt{n}})n}]}
e^{-\frac{n-3}{2}\cdot(\frac{\alpha}{\sqrt{n}}+\frac{k}{n})^2}\label{e3}\\
&\leq&\frac{1}{n}\sum_{k=0}^{\infty} e^{-\frac{n-3}{2n}\cdot{\frac{\alpha^2n+2\alpha \sqrt{n}k+2k-1}{n}}}\nonumber\\
&\leq&\frac{1}{n}\sum_{k=0}^{\infty} e^{-0.4625\cdot \left(\alpha^2-\frac{1}{n}+\frac{2\alpha\sqrt{n}+2}{n}k \right)}~~\left({\rm as}~\frac{n-3}{2n}\ge \frac{40-3}{80}=0.4625\right)\nonumber\\
&= & \frac{e^{-0.4625\left(\alpha^2-\frac{1}{n}\right)}}{n}\sum_{k=0}^{\infty}e^{-0.4625
\left(\frac{2\alpha}{\sqrt{n}}+\frac{2}{n}\right)k}\nonumber\\
&\leq & \frac{e^{-0.4625 \left(\alpha^2-\frac{1}{40}\right)}}{n}\cdot \frac{1}{1-e^{-0.4625\left(\frac{2\alpha}{\sqrt{n}}+\frac{2}{n}\right) }}\nonumber\\
&\leq & \frac{e^{-0.4625 \left(\alpha^2-\frac{1}{40}\right)}}{\sqrt{n}}
\cdot\frac{1}{\sqrt{n}\left(1-e^{-0.4625\frac{2\alpha}{\sqrt{n}}}\right)}  \nonumber\\
&\leq & \frac{e^{0.4625/40}e^{-0.4625\alpha^2}}{\sqrt{n}\sqrt{40}\left(1-e^{-0.4625
\frac{2\alpha}{\sqrt{40}} }\right)}\label{rel:5}~~({\rm as}~n\ge 40)\\
&<& \frac{0.16~e^{-0.4625\alpha^2}}{\sqrt{n}\left(1-~e^{
 -0.1462{\alpha}  }\right)},\label{rel:6}
\end{eqnarray}
where  the inequality (\ref{e2}) holds since $\sqrt{1-t^2}$ is decreasing with respect to $t\in[0,1]$,
the inequality (\ref{e3}) follows from
$
1-x^2<e^{-x^2},~\forall ~x\in(0,1)$,
and the inequality (\ref{rel:5}) holds since the function $w(x)={x\bigg(1-e^{{-0.925}\frac{\alpha}{x}}\bigg)}$
is strictly increasing in terms of $x$, which can be verified by checking the positivity of the derivative of $w(x)$:
\[
w'(x)
= e^{-\frac{0.925\alpha}{x}}
\left(e^{\frac{0.925\alpha}{x}}-
\left(1+\frac{0.925\alpha}{x}\right)\right)>0.
\]

Plugging (\ref{rel:4}) and (\ref{rel:6}) into (\ref{p2:1}) yields
\[
S(n,\alpha)<\frac{ 0.08\alpha~e^{-0.4625\alpha^2} }
{ 0.16\alpha~e^{-0.4625\alpha^2} + \left(1-e^{
 -0.1462{\alpha}}\right)\left(1-e^{-0.475\alpha^2}\right)}.
\]
In order to prove (\ref{ie:1}), it is sufficient to show that the inequality
\[
g(\alpha):= 0.08\alpha~e^{-0.0125\alpha^2}\left(1-2e^{-0.45\alpha^2}\right)<
 h(\alpha):=\left(1-e^{
 -0.1462{\alpha}}\right)\left(1-e^{-0.475\alpha^2}\right)
\]
holds for any $\alpha>0$. Firstly, we have
\begin{eqnarray*}
&&g(\alpha)<0<h(\alpha),~\forall \alpha\in (0,1.24], \\
&&g(\alpha)<0.08\max_{\alpha}\{\alpha~e^{-0.0125\alpha^2}\}<0.32<h(2.8)\le h(\alpha),~
\forall \alpha>2.8.
\end{eqnarray*}
Secondly, since
both $g(\alpha)$ and $h(\alpha)$ are increasing in the interval $[1.24,2.8]$,
we have
\begin{eqnarray*}
&&g(\alpha)\le g(1.8)<h(1.24)\le h(\alpha),~\forall \alpha\in[1.24,1.8],\\
&&g(\alpha)\le g(2.5)<h(1.8)\le h(\alpha),~\forall \alpha\in[1.8,2.5],\\
&&g(\alpha)\le g(2.8)<h(2.5)\le h(\alpha),~\forall \alpha\in[2.5,2.8].
\end{eqnarray*}
The proof of (\ref{ie:1}) in the case of $n\ge 40$ is complete.

(b)
Suppose $n\in\{2,3,\ldots,39\}$. Define
$q(\alpha):=e^{-0.45\alpha^2}$ and
\[
m(\alpha):=\max_{n\in\{2,3,\ldots,39\}} S(n,\alpha).
\]
Since $S(n,\alpha)$ is decreasing in terms of  $\alpha$,
$m(\alpha)$ is also monotonically decreasing.
Let $
\{x_1,x_2,x_3,x_4,x_5,x_6\}=\{0,1.2,2.9,3.8,4.9,6.3\}.$
We can verify that
\[
m(x_k)<q(x_{k+1}), ~k=1,\ldots,5.
\]
According to the monotonicity of $q(\alpha)$ and $m(\alpha)$, we have
\begin{equation}
m(\alpha)\le m(x_k)<q(x_{k+1})\le q(\alpha),
~\forall \alpha\in [x_{k},x_{k+1}],~k=1,\ldots,5.
\end{equation}
Since $(0,\sqrt{39}]\subseteq \cup_{k=1}^5[x_k,x_{k+1}]$ and $m(\alpha)=0$
when $\alpha\ge \sqrt{39}$,
the proof of (\ref{ie:1}) in the case of $n\in\{2,3,\ldots,39\}$ is complete.

\end{document}